\documentclass{amsart}
\usepackage{amsmath}
\usepackage{amssymb}
\usepackage{color}
\usepackage{mathdots}

\usepackage{graphicx}
\usepackage{epsfig}
\usepackage{multirow}
\usepackage{colortbl}

\definecolor{lightgray}{gray}{0.9}

\DeclareMathOperator{\Id}{Id}

\def\FF{\mathbb{F}}
\def\NN{\mathbb{N}}

\newtheorem{theorem}{Theorem}[section]

\newtheorem{proposition}[theorem]{Proposition}
\newtheorem{corollary}[theorem]{Corollary}
\theoremstyle{definition}

\theoremstyle{remark}
\newtheorem{remark}[theorem]{Remark}
\numberwithin{equation}{section}

\def\separa{\hbox to 14 truecm{\hrulefill}}

\author[P. Danchev]{Peter Danchev}
\address{Institute of Mathematics and Informatics, Bulgarian Academy of Sciences, 1113 Sofia, Bulgaria}
\email{danchev@math.bas.bg}
\thanks{The first author was partially supported  by the BIDEB 2221 of T\"UB\'ITAK}

\author[E. Garc\'\i a]{Esther Garc\'\i a}
\address{Departamento de Matem\'{a}tica  Aplicada, Ciencia e Ingenier\'{\i}a de  Materiales y Tecnolog\'{\i}a Electr\'onica,
Universidad Rey Juan Carlos, 28933 M\'{o}s\-to\-les (Madrid), Spain}
\thanks{The second author and third authors were partially supported by Ayuda Puente 2023, URJC, and MTM2017-84194-P (AEI/FEDER, UE)}
\email{esther.garcia@urjc.es}

\author[M. G\'omez Lozano]{Miguel G\'omez Lozano}
\address{Departamento de \'Algebra, Geometr\'{\i}a y
Topolog\'{\i}a, Universidad de M\'alaga, 29071 M\'alaga, Spain}
\thanks{The three authors were partially supported by the Junta de Andaluc\'{\i}a FQM264}
\email{miggl@uma.es}

\begin{document}

\title[Decomposition as sum of diagonalizable and square-zero matrices]{Matrices over finite fields of odd characteristic as sums of diagonalizable and square-zero matrices}
\maketitle

\begin{abstract} Let $\mathbb{F}$ be a finite field of odd characteristic. When $|\mathbb{F}|\ge 5$, we prove that every matrix $A$ admits a decomposition into $D+M$ where $D$ is diagonalizable and $M^2=0$. For  $\mathbb{F}=\mathbb{F}_3$, we show that such decomposition is possible for non-derogatory matrices of order at least 5, and more generally, for matrices whose first invariant factor is not a non-zero trace irreducible polynomial of degree 3; we also establish that matrices consisting of direct sums of companion matrices, all of them associated to the same irreducible polynomial of non-zero trace and degree 3 over $\mathbb{F}_3$, never admit such decomposition.

These results completely settle the question posed by Breaz in Lin. Algebra \& Appl. (2018) asking if it is true that for big enough positive integers $n\ge 3$ all matrices $A$ over a field of odd cardinality $q$ admit decompositions of the form $E+M$ with $E^q=D$ and $M^2=0$: the answer is {\it yes} for $q\ge 5$, but there are counterexamples for $q=3$ and each order $n=3k$, $k\ge 1$.
\end{abstract}


\bigskip




\section{Backgrounds}

All matrices considered in the present paper are square -- they have the same number of columns and rows; such number is called the order of the matrix. Now, for completeness of the exposition and the reader's convenience, we recollect some fundamental material. A matrix is said to be {\it non-derogative} (or, in other words, {\it non-derogatory}) if its minimal polynomial coincides  with the characteristic polynomial of the matrix; in such case, the matrix is similar to the companion matrix of its characteristic polynomial. A matrix $X$ is said to be {\it diagonalizable} (or, in other terms, {\it non-defective}) if it is similar to a diagonal matrix, i.e., there exits an invertible matrix $U$ such that the matrix $U^{-1}XU$ is diagonal. Recall also that a {\it nilpotent} matrix $Y$ is the one for which $Y^k = 0$ for some integer $k>0$. In the case where $k=2$, such a matrix is just called {\it square-zero} for short. Likewise, a matrix $Z$ is known to be a {\it $t$-potent} for some integer $t\geq 2$ provided $Z^t=Z$. The case when $t=2$, such a matrix is just termed {\it idempotent}.

In \cite{Br}, Breaz studies those matrices over finite fields of odd characteristic that are a sum of a potent matrix and a nilpotent matrix, proving that each matrix over a field of odd cardinality $q$ is a sum of a $q$-potent matrix and a nilpotent matrix of order less than or equal to $3$ (see \cite{BM} too, as well as \cite{AM}, \cite{A}, \cite{DGL3} and \cite{DGL4} for a more detailed information). He also asks in \cite[Remark 7]{Br} whether for matrices of an enough big order over a finite field of odd cardinality $q$ there will always exist a decomposition into a $q$-potent matrix plus a square-zero matrix.
Notice that all diagonalizable matrices over a finite field of cardinality $q$ are $q$-potent and, conversely, if $D^q=D\in \mathbb{M}_n(\mathbb{F}_q)$, then $D$ is diagonalizable since its minimal polynomial divides the separable polynomial $x^q-x$.

In \cite{DGL1}, we initiated the study of those matrices which are a sum of a diagonalizable matrix and a square-zero matrix. The main two results of \cite{DGL1}, which we will extend in this article, are the following:

\medskip

\noindent{\bf Lemma} \cite[Lemma 3.1]{DGL1} {\it
Let $\mathbb{F}$ be a field, let $n\ge 3$  and let $A\in \mathbb{M}_n(\mathbb{F})$ be a non-derogative matrix with trace $c$. Then:

\medskip

\noindent$\bullet$ If $c=0$ and $|\mathbb{F}|\ge n$, $A$ admits a decomposition into $D+M$ where $D$ is diagonalizable and $M^2=0$.

\medskip

\noindent$\bullet$ If $c\ne 0$ and $|\mathbb{F}|\ge n+1$, $A$ admits a decomposition into $D+M$, where $D$ is diagonalizable and $M^2=0$.}

\medskip

\noindent{\bf Theorem} \cite[Theorem 3.6]{DGL1} {\it
Given any field $\mathbb{F}$, all matrices in $\mathbb{M}_2(\mathbb{F})$ admit a decomposition into $D+M$, where $D$ is a diagonalizable matrix and $M$ is a matrix such that $M^2=0$.

Let $n\ge 3$ and let $\mathbb{F}$ be a field with $|\mathbb{F}|\ge n+1$. Then every matrix $A\in \mathbb{M}_n(\mathbb{F})$ admits a decomposition into $D+M$, where $D$ is a diagonalizable matrix and $M$ is a matrix such that $M^2=0$. In particular, square matrices over infinite fields always admit such decomposition.}

\medskip

In this paper, we succeed to establish that any non-derogatory matrix of size at least 5 over a field of odd characteristic can be written as the sum of a diagonalizable matrix and a square-zero matrix (see, for more account, Proposition~\ref{cnonderog} and Proposition~\ref{cnonderog3}).

With this result in hand, together with the conclusions of \cite{DGL1} cited above, we can show that for fields $\mathbb{F}$ of odd characteristic and cardinality at least 5, every matrix $A\in\mathbb{M}_n(\mathbb{F})$ can be expressed as $A=D+M$, where $D$ is a diagonalizable matrix and $M$ is a square-zero matrix; in particular, $D^q=D$ and so $A$ can be expressed as the sum of $q$-potent matrix and a square-zero matrix. Moreover, if $\mathbb{F}=\mathbb{F}_3$ and the smallest invariant factor of a matrix  $A\in\mathbb{M}_n(\mathbb{F})$ is not a non-zero trace irreducible polynomial of degree 3, then there exists a square-zero matrix $M\in \mathbb{M}_n(\mathbb{F})$ and a diagonalizable matrix $D$ such that $A=D+M$.

We also show that, for each $n=3k$, where $k\ge 1$, there exist matrices of order $n$ over the field $\mathbb{F}_3$ that cannot be expressed as $D+M$ with $D^3=D$ and $M^2=0$.

We end our work with a plan for a further research treatment of this type of decomposition of matrices over fields of characteristic 2, by raising two parallel questions.

\section{Principal Results}

We begin with a technical result.

\begin{proposition}\label{cnonderog}
Let $\mathbb{F}$ be a field of odd characteristic, and let $A\in\mathbb{M}_n(\mathbb{F})$ be a non-derogative matrix of order $n\ge 5$. The following statements hold:

\begin{enumerate}
\item If $\text{Trace}(A)=u_{n-1}\ne 0$, then there exists a square-zero matrix $M$ such that $A+M$ is diagonalizable with $3$ different eigenvalues $\{0, u_{n-1}, -u_{n-1}\}$.
   \item If $\text{Trace}(A)=0$ and $n$ is odd, then there exists a square-zero matrix $M$ such that $A+M$ is diagonalizable with $3$ different eigenvalues $\{0, 1, -1\}$.
   \item If $\text{Trace}(A)=0$, $n$ is even and $\FF\ne \mathbb{F}_3$, then there exists a square-zero matrix $M$ such that $A+M$ is diagonalizable with at most $4$ different eigenvalues.
\end{enumerate}

In particular, every non-derogatory matrix of order at least 5 over a field of odd cardinality $q\ge 5$ can be written as the sum of a diagonalizable matrix and a square-zero matrix.
\end{proposition}

\begin{proof}
It is clear that if a matrix $A$ can be decomposed as the sum of a nilpotent matrix and a diagonalizable matrix, then any matrix $B$ similar to $A$ also decomposes in the same way. Therefore, we can suppose that $A$ has the following form

$$A=\left(
 \begin{array}{ccccccccc}
 0 & 0 & 0 & 0 & 0 & \cdots & 0 & 0 & u_0 \\
 1 & 0 & 0 & 0 & 0 & \cdots & 0 & 0 & u_1 \\
 0 & 1 & 0 & 0 & 0 & \cdots & 0 & 0 & u_2 \\
 0 & 0 & 1 & 0 & 0 & \cdots & 0 & 0 & u_3 \\
 0 & 0 & 0 & 1 & 0 & \cdots & 0 & 0 & u_4 \\
 0 & 0 & 0 & 0 & 1 & \vdots & 0 & 0 & u_5 \\
 \vdots & \vdots & \vdots & \vdots & \vdots & \ddots & \vdots & \vdots & \vdots \\
 0 & 0 & 0 & 0 & 0 & \cdots & 1 & 0 & u_{n-2} \\
 0 & 0 & 0 & 0 & 0 & \cdots & 0 & 1 & u_{n-1} \\
 \end{array}
 \right)$$
Along this proof, for smoothness of the presentation, all the matrices will be identified with endomorphisms of the vector space $\mathbb{F}^n$ with respect to the canonical basis $\{e_1,\dots, e_n\}$.

\medskip

(1.a) Let us study the case when $\text{Trace}(A)=u_{n-1}\ne 0$ and $n$ is odd: let us consider $M=(m_{ij})$, where $m_{k+1,k}=-1$ and $m_{k-1,k}=u_{n-1}^2$ for $k=3, 5, \dots,n-2$, $m_{1,n}=-u_0$, $m_{k,n}=-u_{k-1}-u_{n-1}u_{k}$ for $k=2,4,\dots, n-3$, $m_{n-1,n}=-u_{n-2}$ and the rest equal to $0$, i.e.,

$$M=\left(
       \begin{array}{ccccccccc}
         0 & 0 & 0 & 0 & 0 & \cdots & 0 & 0 & -u_0 \\
         0 & 0 & u_{n-1}^2 & 0 & 0 & \cdots & 0 & 0 & -u_1-u_{n-1} u_2 \\
         0 & 0 & 0 & 0 & 0 & \cdots & 0 & 0 & 0 \\
         0 & 0 & -1 & 0 & u_{n-1}^2 & \cdots & 0 & 0 & -u_3-u_{n-1}u_4 \\
         0 & 0 & 0 & 0 & 0 & \cdots & 0 & 0 & 0 \\
         0 & 0 & 0 & 0 & -1 & \cdots & 0 & 0 & -u_5-u_{n-1}u_6 \\
         \vdots & \vdots & \vdots & \vdots & \vdots & \ddots & \vdots & \vdots & \vdots \\
         0 & 0 & 0 & 0 & 0 & \cdots & -1 & 0 & -u_{n-2} \\
         0 & 0 & 0 & 0 & 0 & \cdots & 0 & 0 & 0 \\
       \end{array}
     \right).$$

Thus, we calculate that $M^2=0$ and
$$A+M=\left(
       \begin{array}{ccccccccc}
         0 & 0 & 0 & 0 & 0 & \cdots & 0 & 0 & 0 \\
         1 & 0 & u_{n-1}^2 & 0 & 0 & \cdots & 0 & 0 & -u_{n-1}u_2 \\
         0 & 1 & 0 & 0 & 0 & \cdots & 0 & 0 & u_2 \\
         0 & 0 & 0 & 0 & u_{n-1}^2 & \cdots & 0 & 0 & -u_{n-1}u_4 \\
         0 & 0 & 0 & 1 & 0 & \cdots & 0 & 0 & u_4 \\
         \vdots & \vdots & \vdots & \vdots & \vdots & \ddots & \vdots & \vdots & \vdots \\
         0 & 0 & 0 & 0 & 0 & \cdots & 0 & 0 & -u_{n-1}u_{n-3} \\
         0 & 0 & 0 & 0 & 0 & \cdots & 0 & 0 & u_{n-3} \\
         0 & 0 & 0 & 0 & 0 & \cdots & 0 & 0 & 0 \\
         0 & 0 & 0 & 0 & 0 & \cdots & 0 & 1 & u_{n-1} \\
       \end{array}
     \right).$$
With respect to the new basis $$B'=\{e_1, e_2,\dots,\, e_{n-2},u_{n-1}e_{n-1}+\sum_{i=1}^{(n-3)/2} u_{2i}e_{2i},\, u_{n-1}e_n+\sum_{i=1}^{(n-3)/2} u_{2i}e_{2i+1}\},$$ we have that
$$(A+M)_{B'}=\left(
       \begin{array}{ccc|cc|cc|cc}
         0 & 0 & 0 & 0 & 0 & \cdots & 0 & 0 & 0 \\
         1 & 0 & u_{n-1}^2 & 0 & 0 & \cdots & 0 & 0 & 0 \\
         0 & 1 & 0 & 0 & 0 & \cdots & 0 & 0 & 0 \\
         \hline
         0 & 0 & 0 & 0 & u_{n-1}^2 & \cdots & 0 & 0 & 0 \\
         0 & 0 & 0 & 1 & 0 & \cdots & 0 & 0 & 0 \\
         \hline
         0 & 0 & 0 & 0 & 0 & \cdots &  & 0 & 0 \\
         \vdots & \vdots & \vdots & \vdots & \vdots & \ddots & \vdots & \vdots & \vdots \\
         \hline
         0 & 0 & 0 & 0 & 0 & \cdots & 0 & 0 & 0 \\
         0 & 0 & 0 & 0 & 0 & \cdots & 0 & 1 & u_{n-1} \\
       \end{array}
     \right),$$
which consists of the direct sum of a block of the form $\left(
                                               \begin{array}{ccc}
                                                 0 & 0 & 0 \\
                                                 1 & 0 & u_{n-1}^2 \\
                                                 0 & 1 & 0 \\
                                               \end{array}
                                             \right)$
that satisfies the polynomial $x^3-u_{n-1}^2x$, a block of the form $\left(
\begin{array}{ccc}
0 & 0 \\
1 & u_{n-1}\\
\end{array}
\right)$ that satisfies the polynomial $x^2-u_{n-1}x$
and, if $n\ge 7$, blocks of the form $\left(
                                               \begin{array}{ccc}
                                                 0 & u_{n-1}^2 \\
                                                 1 & 0\\
                                               \end{array}
                                             \right)$ that satisfy the polynomial $x^2-u_{n-1}^2$.

Therefore, the matrix $A+M$ satisfies the equation $(x-u_{n-1})(x+u_{n-1})x=0$, and so it is diagonalizable with eigenvalues $\{0,u_{n-1},-u_{n-1}\}$.

\medskip

(1.b) Let us study the case when $\text{Trace}(A)=u_{n-1}\ne 0$ and $n$ is even: let us consider $M=(m_{ij})$, where $m_{k+1,k}=-1$ and $m_{k-1,k}=u_{n-1}^2$ for $k=2, 4, \dots,n-2$,  $m_{k,n}=-u_{k-1}-u_{n-1} u_k$ for $k=1, 3, \dots,n-3$, $m_{n-1,n}=-u_{n-2}$  and the rest equal to $0$, i.e.,

$$
M=\left(
       \begin{array}{ccccccccc}
         0 & u_{n-1}^2 & 0 & 0 & 0 & \cdots & 0 & 0 & -u_0-u_{n-1}u_1 \\
         0 & 0 & 0 & 0 & 0 & \cdots & 0 & 0 & 0 \\
         0 & -1 & 0 & u_{n-1}^2 & 0 & \cdots & 0 & 0 & -u_2-u_{n-1}u_3 \\
         0 & 0 & 0 & 0 & 0 & \cdots & 0 & 0 & 0 \\
         0 & 0 & 0 & -1 & 0 & \cdots & 0 & 0 & -u_4-u_{n-1}u_5 \\
         \vdots & \vdots & \vdots & \vdots & \vdots & \ddots & \vdots & \vdots & \vdots \\
         0 & 0 & 0 & 0 & 0 & \cdots & u_{n-1}^2 & 0 & -u_{n-4}-u_{n-1}u_{n-3} \\
         0 & 0 & 0 & 0 & 0 & \cdots & 0 & 0 & 0 \\
         0 & 0 & 0 & 0 & 0 & \cdots & -1 & 0 & -u_{n-2} \\
         0 & 0 & 0 & 0 & 0 & \cdots & 0 & 0 & 0 \\
       \end{array}
     \right)
$$
Thus, we compute that $M^2=0$ and
$$A+M=\left(
       \begin{array}{cccccccccc}
         0 & u_{n-1}^2 & 0 & 0 & 0 & \cdots &0& 0 & 0 & -u_{n-1} u_1 \\
         1 & 0 & 0 & 0 & 0 & \cdots &0& 0 & 0 & u_1 \\
         0 & 0 & 0 & u_{n-1}^2 & 0 & \cdots &0& 0 & 0 & -u_{n-1}u_3 \\
         0 & 0 & 1 & 0 & 0 & \cdots &0& 0 & 0 & u_3 \\
         0 & 0 & 0 & 0 & 0 & \cdots &0& 0 & 0 & -u_{n-1}u_5 \\
         \vdots & \vdots & \vdots & \vdots & \vdots & \ddots & \vdots & \vdots & \vdots &\vdots \\
         0 & 0 & 0 & 0 & 0 & \cdots &0& u_{n-1}^2 & 0 & -u_{n-1} u_{n-3} \\
         0 & 0 & 0 & 0 & 0 & \cdots &1& 0 & 0 & u_{n-3} \\
         0 & 0 & 0 & 0 & 0 & \cdots &0& 0 & 0 & 0 \\
         0 & 0 & 0 & 0 & 0 & \cdots &0& 0 & 1 & u_{n-1} \\
       \end{array}
     \right).$$
With respect to the  basis $$B'=\{e_1, e_2,\dots, e_{n-2},\, u_{n-1}e_{n-1}+\sum_{i=1}^{(n-2)/2} u_{2i-1} e_{2i-1},\, u_{n-1}e_{n}+\sum_{i=1}^{(n-2)/2} u_{2i-1} e_{2i}\},$$ we deduce that
$$(A+M)_{B'}=\left(
       \begin{array}{cc|cc|cc|cc|cc}
         0 & u_{n-1}^2 & 0 & 0 & 0 & \cdots &0& 0 & 0 & 0 \\
         1 & 0 & 0 & 0 & 0 & \cdots &0& 0 & 0 & 0 \\
         \hline
         0 & 0 & 0 & u_{n-1}^2 & 0 & \cdots &0& 0 & 0 & 0 \\
         0 & 0 & 1 & 0 & 0 & \cdots &0& 0 & 0 & 0 \\
         \hline
         0 & 0 & 0 & 0 & 0 & \cdots &0& 0 & 0 & 0 \\
         \vdots & \vdots & \vdots & \vdots & \vdots & \ddots & \vdots & \vdots & \vdots \\
         \hline
         0 & 0 & 0 & 0 & 0 & \cdots &0& u_{n-1}^2 & 0 & 0 \\
         0 & 0 & 0 & 0 & 0 & \cdots &1& 0 & 0 & 0 \\
         \hline
         0 & 0 & 0 & 0 & 0 & \cdots &0& 0 & 0 & 0 \\
         0 & 0 & 0 & 0 & 0 & \cdots &0& 0 & 1 & u_{n-1} \\
       \end{array}
     \right),$$
which is the direct sum of blocks of the form $\left(\begin{array}{ccc}
                                                 0 & u_{n-1}^2 \\
                                                 1 & 0\\
                                               \end{array}
                                             \right)$ that satisfies $x^2-u_{n-1}^2$ and
a block of the form $\left(\begin{array}{ccc}
                                                 0 & 0 \\
                                                 1 & u_{n-1}\\
                                               \end{array}
                                             \right)$ that satisfies $x^2-u_{n-1}x$.

Therefore, the matrix $A+M$ satisfies the equation $(x-u_{n-1})(x+u_{n-1})x=0$, and so it is diagonalizable with eigenvalues $\{0,u_{n-1},-u_{n-1}\}$.

\medskip

(2) Let us study the case when $\text{Trace}(A)=0$ and $n$ odd: let us consider $M=(m_{ij})$, where $m_{k+1,k}=-1$ and $m_{k-1,k}=1$ for $k=2, 4, \dots,n-1$ and the rest equal to $0$, i.e.,

$$M=\left(
       \begin{array}{ccccccccc}
         0 & 1 & 0 & 0 & 0 & \cdots & 0 & 0 & 0 \\
         0 & 0 & 0 & 0 & 0 & \cdots & 0 & 0 & 0 \\
         0 & -1 & 0 & 1 & 0 & \cdots & 0 & 0 & 0 \\
         0 & 0 & 0 & 0 & 0 & \cdots & 0 & 0 & 0 \\
         0 & 0 & 0 & -1 & 0 & \cdots & 0 & 0 & 0 \\
         \vdots & \vdots & \vdots & \vdots & \vdots & \ddots & \vdots & \vdots & \vdots \\
         0 & 0 & 0 & 0 & 0 & \cdots & 0 & 1 & 0 \\
         0 & 0 & 0 & 0 & 0 & \cdots & 0 & 0 & 0 \\
         0 & 0 & 0 & 0 & 0 & \cdots & 0 & -1 & 0 \\
       \end{array}
     \right)
$$
Thus, we find that $M^2=0$ and
$$A+M=\left(
       \begin{array}{ccccccccc}
         0 & 1 & 0 & 0 & 0 & \cdots & 0 & 0 & u_0 \\
         1 & 0 & 0 & 0 & 0 & \cdots & 0 & 0 & u_1 \\
         0 & 0 & 0 & 1 & 0 & \cdots & 0 & 0 & u_2 \\
         0 & 0 & 1 & 0 & 0 & \cdots & 0 & 0 & u_3 \\
         0 & 0 & 0 & 0 & 0 & \cdots & 0 & 0 & u_4 \\
         \vdots & \vdots & \vdots & \vdots & \vdots & \ddots & \vdots & \vdots & \vdots \\
         0 & 0 & 0 & 0 & 0 & \cdots & 0 & 1 & u_{n-3} \\
         0 & 0 & 0 & 0 & 0 & \cdots & 1 & 0 & u_{n-2} \\
         0 & 0 & 0 & 0 & 0 & \cdots & 0 & 0 & 0 \\
       \end{array}
     \right).$$
With respect to the new basis $$B'=\{e_1, e_2,\dots, e_{n-1}, e_n-\sum_{i=1}^{(n-1)/2} u_{2i-2} e_{2i}-\sum_{i=1}^{(n-1)/2}u_{2i-1}e_{2i-1}\},$$ we derive that
$$(A+M)_{B'}=\left(
       \begin{array}{cc|cc|ccc|cc}
         0 & 1 & 0 & 0 & 0 & \cdots & 0 & 0 & 0 \\
         1 & 0 & 0 & 0 & 0 & \cdots & 0 & 0 & 0 \\
         \hline
         0 & 0 & 0 & 1 & 0 & \cdots & 0 & 0 & 0 \\
         0 & 0 & 1 & 0 & 0 & \cdots & 0 & 0 & 0 \\
         \hline
         0 & 0 & 0 & 0 & 0 & \cdots & 0 & 0 & 0 \\
         \vdots & \vdots & \vdots & \vdots & \vdots & \ddots & \vdots & \vdots & \vdots \\
         0 & 0 & 0 & 0 & 0 & \cdots & 0 & 1 & 0 \\
         \hline
         0 & 0 & 0 & 0 & 0 & \cdots & 1 & 0 & 0 \\
         0 & 0 & 0 & 0 & 0 & \cdots & 0 & 0 & 0 \\
       \end{array}
     \right),$$
which is the direct sum of blocks of the form $\left(
                                               \begin{array}{ccc}
                                                 0 & 1 \\
                                                 1 & 0\\
                                               \end{array}
                                             \right)$ that satisfy $x^2-1$ and a zero block of order one.
Therefore, $A+M$ satisfies the equation $(x-1)(x+1)x=0$, and so it is diagonalizable with eigenvalues $\{0,1,-1\}$.

\medskip

(3) Finally, let us study the case when $\text{Trace}(A)=0$ and $n$ is even with $\FF\ne \mathbb{Z}_3$: Since $|\FF|\ge 5$, let us consider $0\ne \alpha\in \FF$ such that $\alpha^2\ne 1$, and also let us consider $M=(m_{ij})$, where $m_{k+1,k}=-1$ and $m_{k-1,k}=1$ for $k=2, 4, \dots,n-2$, $m_{k,n}=-u_{k-1}$ for $k=1, 3, \dots,n-3$, $m_{n-1,n}=\alpha^2-u_{n-2}$ and the rest equal to $0$, i.e.,

$$
M=\left(
       \begin{array}{cccccccccc}
         0 & 1 & 0 & 0 & 0 & \cdots & 0& 0 & 0 & -u_0 \\
         0 & 0 & 0 & 0 & 0 & \cdots & 0& 0 & 0 & 0 \\
         0 & -1 & 0 & 1 & 0 & \cdots & 0& 0 & 0 & -u_2 \\
         0 & 0 & 0 & 0 & 0 & \cdots & 0& 0 & 0 & 0 \\
         0 & 0 & 0 & -1 & 0 & \cdots & 0& 0 & 0 & -u_4 \\
         \vdots & \vdots & \vdots & \vdots &\vdots & \ddots & \vdots & \vdots & \vdots & \vdots \\
         0 & 0 & 0 & 0 & 0 & \cdots &0 & 1 & 0 & -u_{n-4} \\
         0 & 0 & 0 & 0 & 0 & \cdots &0 & 0 & 0 & 0 \\
         0 & 0 & 0 & 0 & 0 & \cdots &0 & -1 & 0 & \alpha^2-u_{n-2} \\
         0 & 0 & 0 & 0 & 0 & \cdots & 0& 0 & 0 & 0 \\
       \end{array}
     \right)$$
Thus, we have that $M^2=0$ and
$$A+M=\left(
       \begin{array}{cccccccccc}
         0 & 1 & 0 & 0 & 0 & \cdots & 0& 0 & 0 & 0 \\
         1 & 0 & 0 & 0 & 0 & \cdots & 0& 0 & 0 & u_1 \\
         0 & 0 & 0 & 1 & 0 & \cdots & 0& 0 & 0 & 0 \\
         0 & 0 & 1 & 0 & 0 & \cdots & 0& 0 & 0 & u_3 \\
         0 & 0 & 0 & 0 & 0 & \cdots & 0& 0 & 0 & 0 \\
         \vdots & \vdots & \vdots & \vdots & \vdots & \ddots & \vdots& \vdots & \vdots & \vdots \\
         0 & 0 & 0 & 0 & 0 & \cdots & 0& 1 & 0 & 0 \\
         0 & 0 & 0 & 0 & 0 & \cdots & 1& 0 & 0 & u_{n-3} \\
         0 & 0 & 0 & 0 & 0 & \cdots & 0& 0 & 0 & \alpha^2 \\
         0 & 0 & 0 & 0 & 0 & \cdots & 0& 0 & 1 & 0 \\
       \end{array}
     \right).$$
With respect to the basis $$B'=\{e_1, e_2,\dots, e_{n-2},\, e_{n-1}+\sum_{i=1}^{(n-2)/2} {u_{2i-1}\over \alpha^2-1} e_{2i-1}, \, e_{n}+\sum_{i=1}^{(n-2)/2} {u_{2i-1}\over \alpha^2-1} e_{2i}\},$$ we obtain that
$$(A+M)_{B'}=\left(
       \begin{array}{cc|cc|cc|cc|cc}
         0 & 1 & 0 & 0 & 0 & \cdots & 0& 0 & 0 & 0 \\
         1 & 0 & 0 & 0 & 0 & \cdots & 0& 0 & 0 & 0 \\
         \hline
         0 & 0 & 0 & 1 & 0 & \cdots & 0& 0 & 0 & 0 \\
         0 & 0 & 1 & 0 & 0 & \cdots & 0& 0 & 0 & 0 \\
         \hline
         0 & 0 & 0 & 0 & 0 & \cdots & 0& 0 & 0 & 0 \\
         \vdots & \vdots & \vdots & \vdots & \vdots & \ddots & \vdots& \vdots & \vdots & \vdots \\
         \hline
         0 & 0 & 0 & 0 & 0 & \cdots & 0& 1 & 0 & 0 \\
         0 & 0 & 0 & 0 & 0 & \cdots & 1& 0 & 0 & 0 \\
         \hline
         0 & 0 & 0 & 0 & 0 & \cdots & 0& 0 & 0 & \alpha^2 \\
         0 & 0 & 0 & 0 & 0 & \cdots & 0& 0 & 1 & 0 \\
       \end{array}
     \right),$$
which consists of the direct sum of blocks of the form  $\left(\begin{array}{ccc}
                                                 0 & 1 \\
                                                 1 & 0\\
                                               \end{array}
                                             \right)$ that satisfies $x^2-1$ and
a block of the form $\left(\begin{array}{ccc}
                                                 0 & \alpha^2 \\
                                                 1 & 0\\
                                               \end{array}
                                             \right)$ that satisfies $x^2-\alpha^2$.

Therefore, $A+M$ satisfies the equation $(x-\alpha)(x+\alpha)(x-1)(x+1)=0$, and so it is diagonalizable with eigenvalues $\{1,-1,\alpha,-\alpha\}$, completing the arguments after all.
\end{proof}

It was shown in \cite[\S1]{DGL1} that matrices of order 2 can always be decomposed into a diagonalizable matrix and a square-zero matrix. This result, together with Proposition \ref{cnonderog}  lead to the following theorem, which solves positively the question posed by Breaz in \cite[Remark 7]{Br} for fields of odd cardinality at least 5.

\begin{theorem}
Let $\mathbb{F}$ be a finite field of odd cardinality $q$ with $q\ge 5$, and let $n\in \mathbb{N}$. Then, every matrix  $A\in\mathbb{M}_n(\mathbb{F})$ can be expressed as $A=D+M$, where $D$ is a diagonalizable matrix and $M$ is a square-zero matrix. In particular, $D^q=D$, so $A$ can be expressed as the sum of $q$-potent matrix and a square-zero matrix.
\end{theorem}

\begin{proof}
Let $A\in\mathbb{M}_n(\mathbb{F})$ and suppose that $f_1(x)|\dots|f_k(x)$ are its invariant factors; in view of the rational canonical decomposition, $A$ is similar to the direct sum of the companion matrices of $f_1(x),\dots, f_k(x)$.

\medskip

\noindent -- The companion matrices of polynomials of degree less than or equal to $4$ can be decomposed as the sum of a diagonalizable matrix and a square-zero matrix in accordance with the arguments of \cite{DGL1}: those of order 1 are trivially decomposed using the zero matrix; those of order 2, 3 and 4 can always be decomposed thanks to \cite[Theorem 3.6]{DGL1}, because the cardinality of the field $q$ is at least 5 by hypothesis.

\medskip

\noindent -- The companion matrices of polynomials of degree at least 5 can be decomposed as the sum of a diagonalizable matrix and a square-zero matrix according to Proposition \ref{cnonderog}.
\end{proof}

\section{Matrices over the field $\mathbb{F}_3$}

In the previous section we have completely solved the problem of decomposing matrices into diagonalizable and square-zero for fields of odd cardinality at least five. In this section, let us consider matrices over $\mathbb{F}_3$. Notice that in Proposition \ref{cnonderog} we have not included non-derogative matrices of even order and zero trace over $\mathbb{F}_3$. So, let us study them now in what follows.

\begin{proposition}\label{cnonderog3}
Let $A\in\mathbb{M}_n(\mathbb{F}_3)$ be a non-derogative matrix of even order $n$ and  zero trace. Then, there exists a square-zero matrix $M$ such that $A+M$ is diagonalizable.
\end{proposition}

\begin{proof}
If the order $n$ of $A$ is even and not divisible by $3$, we can consider the matrix $A-I$, which has non-zero trace. So, Proposition \ref{cnonderog}(1) applies to get a square-zero matrix $M$ such that $A-I+M=D$ is diagonalizable, whence $A+M=D+I$ is diagonalizable as well.

Assume from now on that $n$ is even and a multiple of $3$, i.e., $n=6k$ for $k\in \NN$. Let us suppose that  $$A=\left(
\begin{array}{ccccccccc}
 0 & 0 & 0 & 0 & 0 & \cdots & 0 & 0 & u_0 \\
 1 & 0 & 0 & 0 & 0 & \cdots & 0 & 0 & u_1 \\
 0 & 1 & 0 & 0 & 0 & \cdots & 0 & 0 & u_2 \\
 0 & 0 & 1 & 0 & 0 & \cdots & 0 & 0 & u_3 \\
 0 & 0 & 0 & 1 & 0 & \cdots & 0 & 0 & u_4 \\
 0 & 0 & 0 & 0 & 1 & \vdots & 0 & 0 & u_5 \\
 \vdots & \vdots & \vdots & \vdots & \vdots & \ddots & \vdots & \vdots & \vdots \\
 0 & 0 & 0 & 0 & 0 & \cdots & 1 & 0 & u_{n-2} \\
 0 & 0 & 0 & 0 & 0 & \cdots & 0 & 1 & 0 \\
\end{array}
\right),$$ and consider $M=(m_{ij})$, where $m_{k+1,k}=-1$ and $m_{k-1,k}=1$ for $k=2, 4, \dots,n-4$, $m_{n-1,n-1}=m_{n-2,n-1}=1$, $m_{n-1,n-2}=m_{n-2,n-2}=-1$, $m_{k,n}=-u_{k-1}-u_k$ for $k=1, 3, \dots,n-5$, $m_{n-3,n}=u_{n-4}$, $m_{n-2,n}=m_{n-1,n}=u_{n-2}$ and the rest equal to $0$, i.e.,

$$M=\left(
\begin{array}{cccccccccccc}
 0 & 1 & 0 & 0 & 0 & \cdots & 0 & 0 &0&0&0& -u_0-u_1 \\
 0 & 0 & 0 & 0 & 0 & \cdots & 0 & 0 &0&0&0& 0 \\
 0 & -1 & 0 & 1 & 0 & \cdots & 0 & 0 &0&0&0& -u_2-u_3 \\
 0 & 0 & 0 & 0 & 0 & \cdots & 0 & 0 &0&0&0& 0 \\
 0 & 0 & 0 & -1 & 0 & \cdots & 0 & 0 &0&0&0& -u_4-u_5 \\
 0 & 0 & 0 & 0 & 0 & \vdots & 0 & 0 &0&0&0& 0 \\
 \vdots & \vdots & \vdots & \vdots & \vdots & \ddots &\vdots & \vdots & \vdots& \vdots & \vdots & \vdots \\
 0 & 0 & 0 & 0 & 0 & \cdots & 0 & 1 &0&0&0& -u_{n-6}-u_{n-5} \\
 0 & 0 & 0 & 0 & 0 & \cdots & 0 & 0 &0&0&0& 0 \\
 0 & 0 & 0 & 0 & 0 & \cdots & 0 & -1 &0&0&0& -u_{n-4} \\
 0 & 0 & 0 & 0 & 0 & \cdots & 0 & 0 &0&-1&1& -u_{n-2} \\
 0 & 0 & 0 & 0 & 0 & \cdots & 0 & 0 &0&-1&1& -u_{n-2} \\
 0 & 0 & 0 & 0 & 0 & \cdots & 0 & 0 &0&0&0& 0 \\
\end{array}
\right).$$
Thus, we have that $M^2=0$ and, if we consider $A+M$, we obtain
$$A+M=\left(
 \begin{array}{cccccccccccc}
 0 & 1 & 0 & 0 & 0 & \cdots & 0 & 0 &0&0&0& -u_1 \\
 1 & 0 & 0 & 0 & 0 & \cdots & 0 & 0 &0&0&0& u_1 \\
 0 & 0 & 0 & 1 & 0 & \cdots & 0 & 0 &0&0&0& -u_3 \\
 0 & 0 & 1 & 0 & 0 & \cdots & 0 & 0 &0&0&0& u_3 \\
 0 & 0 & 0 & 0 & 0 & \cdots & 0 & 0 &0&0&0& -u_5 \\
 0 & 0 & 0 & 0 & 1 & \vdots & 0 & 0 &0&0&0& u_5 \\
 \vdots & \vdots & \vdots & \vdots & \vdots & \ddots &\vdots & \vdots & \vdots& \vdots & \vdots & \vdots \\
 0 & 0 & 0 & 0 & 0 & \cdots & 0 & 1 &0&0&0& -u_{n-5} \\
 0 & 0 & 0 & 0 & 0 & \cdots & 1 & 0 &0&0&0& u_{n-5} \\
 0 & 0 & 0 & 0 & 0 & \cdots & 0 & 0 &0&0&0& 0 \\
 0 & 0 & 0 & 0 & 0 & \cdots & 0 & 0 &1&-1&1& -u_{n-3}-u_{n-2} \\
 0 & 0 & 0 & 0 & 0 & \cdots & 0 & 0 &0&0&1& 0 \\
 0 & 0 & 0 & 0 & 0 & \cdots & 0 & 0 &0&0&1& 0 \\
\end{array}
\right).$$
With respect to the basis
\begin{align*}B'&=\{e_1, e_2,\dots, e_{n-2},\\ &e_{n-1}+\sum_{i=1}^{(n-4)/2} -u_{2i-1} e_{2i}+(u_{n-3}-u_{n-2})e_{n-3}+\frac{(u_{n-3}-u_{n-2}+1)}{2}e_{n-2},\\& e_n +\sum_{i=1}^{(n-4)/2} u_{2i-1} (e_{2i}-e_{2i-1})+(u_{n-2}-u_{n-3})e_{n-3}\},\end{align*}
we get
$$(A+M)_{B'}=\left(
\begin{array}{cc|cc|cc|c|cc|cc|cc}
 0 & 1 & 0 & 0 & 0 &0& \cdots & 0 & 0 &0&0&0& 0 \\
 1 & 0 & 0 & 0 & 0 &0& \cdots & 0 & 0 &0&0&0& 0 \\
 \hline
 0 & 0 & 0 & 1 & 0 &0& \cdots & 0 & 0 &0&0&0& 0 \\
 0 & 0 & 1 & 0 & 0 &0& \cdots & 0 & 0 &0&0&0& 0 \\
 \hline
 0 & 0 & 0 & 0 & 0 &0& \cdots & 0 & 0 &0&0&0& 0 \\
 0 & 0 & 0 & 0 & 1 &0& \vdots & 0 & 0 &0&0&0& 0 \\
 \hline
 \vdots & \vdots & \vdots &\vdots & \vdots & \vdots & \ddots &\vdots & \vdots & \vdots& \vdots & \vdots & \vdots \\
 \hline
 0 & 0 & 0 & 0 & 0 &0& \cdots & 0 & 1 &0&0&0& 0 \\
 0 & 0 & 0 & 0 & 0 &0& \cdots & 1 & 0 &0&0&0& 0 \\
 \hline
 0 & 0 & 0 & 0 & 0 &0& \cdots & 0 & 0 &0&0&0& 0 \\
 0 & 0 & 0 & 0 & 0 &0& \cdots & 0 & 0 &1&-1&0& 0 \\
 \hline
 0 & 0 & 0 & 0 & 0 &0 &  \cdots  & 0 &0&0&0&1& 0 \\
 0 & 0 & 0 & 0 & 0 &0& \cdots  & 0 &0&0&0&1& 0 \\
\end{array}
\right),$$
which consists of the direct sum of blocks of the form  $\left(
                                               \begin{array}{ccc}
                                                 0 & 1 \\
                                                 1 & 0\\
                                               \end{array}
                                             \right)$ that satisfy $x^2-1$,
a block of the form  $\left(
                                               \begin{array}{ccc}
                                                 0 & 0 \\
                                                 1 & -1\\
                                               \end{array}
                                             \right)$ that satisfies $x^2+x$, and a block of the form  $\left(
                                               \begin{array}{ccc}
                                                 1 & 0 \\
                                                 1 & 0\\
                                               \end{array}
                                             \right)$ that satisfies $x^2-x$.

Therefore, $A+M$ satisfies the equation$(x-1)(x+1)x=0$, and so it is diagonalizable with eigenvalues $\{0,1,-1\}$, as desired.
\end{proof}

\begin{remark}\label{counterexampleBreaz}
It is a well-known fact by \cite[Example 6]{Br} that the companion matrices of  non-zero trace irreducible polynomials of degree 3 cannot be decomposed into diagonalizable and zero-square.  Thanks to the previous results, these are the only possible counterexamples for non-derogatory matrices as the next consequence demonstrates.
\end{remark}

\begin{corollary}\label{nonderogatoryF3} Every non-derogatory matrix over $\mathbb{F}_3$ can be decomposed as the sum of a diagonalizable matrix and a square-zero matrix unless it is similar to the companion matrix of a degree 3 irreducible polynomial of non-zero trace.
\end{corollary}

\begin{proof}
For matrices of order 1, 2 and 4, the claim follows directly from \cite[Theorem 3.6, Proposition 5.4]{DGL1}. Non-derogative matrices of order 3 and zero trace can be decomposed in this way by \cite[Lemma 3.1]{DGL1}.
For non-derogative matrices of order at least 5, the decomposition follows with the aid of Propositions \ref{cnonderog}(1)(2) and \ref{cnonderog3}.
\end{proof}

We can extend this result to general matrices over $\mathbb{F}_3$ unless we cannot elude the counterexample mentioned in Remark \ref{counterexampleBreaz}.

\begin{corollary}
Let $A\in\mathbb{M}_n(\mathbb{F}_3)$ and suppose that the first invariant factor of $A$ is not a degree 3 irreducible polynomial of non-zero trace. Then, there exists a square-zero matrix $M$ and a diagonalizable matrix $D$  such that $A=D+M$.
\end{corollary}

\begin{proof}
Suppose that $f_1(x)|\dots| f_k(x)$ are the invariant factors of $A$; by hypothesis,  $f_1(x)$ is not a degree 3 irreducible polynomial of non-zero trace. This condition implies that non of the invariant factors are degree 3 irreducible polynomials of non-zero trace, so their associated companion matrices can all of them be decomposed into a diagonalizable matrix and square-zero matrix by Corollary \ref{nonderogatoryF3}.
\end{proof}

In the following proposition we show that direct sums of companion matrices, all of them associated to the same irreducible degree 3 polynomial with non-zero trace, cannot be decomposed as the sum of a diagonalizable matrix and a square-zero matrix, providing counterexamples to the decomposition into diagonalizable and square-zero for matrices of each order $n=3k$, $k\ge 1$.

\begin{proposition}\label{counterexamplehighorders}
Let  $p(x)\in \mathbb{F}_3[x]$ be an irreducible polynomial of degree $3$ and non-zero trace and let   $A\in\mathbb{M}_n(\mathbb{F}_3)$ be a matrix which is similar to a direct sum of companion matrices all of them associated to $p(x)$.  Then, there does not exist a square-zero matrix $M$ such that $A+M$ is diagonalizable.
\end{proposition}

\begin{proof}
We first observe that over $\mathbb{F}_3$ there are six irreducible polynomials of degree 3 and non-zero trace. Let us suppose that $A\in\mathbb{M}_n(\mathbb{F})$ is similar to a direct sum of companion matrices all of them associated to the polynomial $x^3-x^2-x-1\in \mathbb{F}[x]$, and let us show that for such $A$ there does {\it not exist} a square-zero matrix $M\in \mathbb{M}_n(\mathbb{F})$ such that $A+M$ is diagonalizable. Once we have shown that fact, it is direct to see that there does not exist such $M$ for the matrices $A+\Id$, $A-\Id$, $2A$, $2A+\Id$, $2A-\Id$, which are similar to direct sums of companion matrices, all of them associated to the other five irreducible polynomials of degree 3 and non-zero trace over $\mathbb{F}_3$.

Suppose that there exists a square-zero matrix $M\in\mathbb{M}_n(\mathbb{F}_3)$ such that $D=A+M$ is diagonalizable. Then, we have that $A^3=A^2+A+{\rm Id}$, $M^2=0$ and $D^3=D$ (because every diagonal matrix over $\mathbb{F}_3$ satisfies the equation $x^3-x=0$).

\medskip

(1) Let us prove that $MA^2M=-MAMAM$: in fact, one calculates that
\begin{align*}
0&=M\left((A+M)^3-(A+M)\right)M=MA^3M+MAMAM-MAM\\&=M(A^2+A+{\rm Id})M+MAMAM-MAM\\&=MA^2M+MAMAM.
\end{align*}

(2) Let us prove that $MAMAMAM=MAM+MAMAM$: in fact, one computes that
\begin{align*}
0&=M\left((A+M)^4-(A+M)^2\right)M\\&=MA^4M+MA^2MAM+MAMA^2M-MA^2M\\&=M({\rm Id}+2A+2A^2)M-2MAMAMAM-MA^2M\\&=2MAM-MAMAM+MAMAMAM
\end{align*}

(3) Let us prove that $MAMAM=-MAM$: in fact, one has that
\begin{align*}
0&=M\left((A+M)^5-(A+M)\right)M\\&=MA^5M+MA^3MAM+MA^2MA^2M\\&+MAMA^3M+MAMAMAM-MAM\\&=M(4A^2+2{\rm Id})M+
  M(A^2+A+{\rm Id})MAM+MA^2MA^2M\\&+MAM(A^2+A+{\rm Id})M+MAM+MAMAM-MAM\\&=-MAMAM
  -MAMAMAM+MAMAM+MAMAMAMAM\\&-MAMAMAM+MAMAM+MAMAM\\&
  =-2MAM+MAMAM+MAMAMAM\\&=-MAM+2MAMAM.
\end{align*}

(4) Let us prove that $MAM=0=MA^2M$: in fact, one finds that
$$MAMAMAM=_{(2)}MAM+MAMAM=_{(3)}MAM-MAM=0,$$ so
$$MAMAM=_{(3)}-MAMAMAM=0$$ and, combining (3) with (1), we get
$$MAM=0=-MA^2M.$$

(5) Let us prove that $M=0$, which is an obvious contradiction, because $A$ is not diagonalizable ($A^3\ne A$): indeed, it must be that
\begin{align*}
0&=\left((A+M)^3-(A+M)\right)M=A^3M+AMAM+MA^2M-AM\\&=(A^2+A+{\rm Id})M-AM=A^2M-M=(A^2+{\rm Id})M.
\end{align*}

Now, since $A^2+{\rm Id}$ is invertible, we conclude that $M=0$, as suspected.
\end{proof}

\begin{remark}
Thanks to Proposition \ref{counterexamplehighorders}, we can answer in the negative the question posed by Breaz in \cite[Remark 7]{Br}, namely for the field $\mathbb{F}_3$ it is {\it not} true that for big enough $n$ every matrix of order $n$ can be decomposed into $D+M$ with $D^3=D$ and $M^2=0$.
\end{remark}

\section{Further works}

In this paper, we have dealt with matrices over fields of odd cardinality. It remains to study the analogous problems for fields of characteristic two. To that end, we know from \cite{BCDM} that every matrix in $\mathbb{M}_n(\mathbb{F}_2)$ is a sum of an idempotent matrix and a nilpotent matrix. Nevertheless, \v{S}ter provided in \cite[answer to Question 4.1]{S} two companion matrices of order 4 over $\mathbb{F}_2$ that cannot be expressed as $D+M$, where $D^2=D$ (equivalently, $D$ is diagonalizable) and $M^2=0$. These counterexamples were extended by Shitov in \cite{Sh}, showing that the direct sum of an odd number of such matrices cannot be expressed as $D+M$.

Thereby, it remains to study matrices over fields of characteristic two and cardinality at least four in the way indicated via the following two parallel queries.

\medskip

\noindent{\bf Questions}. Is it true that each matrix in $\mathbb{M}_n(\mathbb{F}_{2^i})$, $i\ge 2$, is a sum of a diagonalizable matrix and a square-zero matrix? If not, is that decomposition valid at least for non-derogative matrices?

\bigskip

\noindent{\bf Funding:} The first-named author (Peter Danchev) was supported in part by the BIDEB 2221 of T\"UB\'ITAK; the second and third-named authors (Esther Garc\'{\i}a and Miguel G\'omez Lozano) were partially supported by Ayuda Puente 2023, URJC, and MTM2017-84194-P (AEI/FEDER, UE). The all three authors were partially supported by the Junta de Andaluc\'{\i}a FQM264.

\vskip3.0pc

\bibliographystyle{plain}

\end{document}